\RequirePackage{fix-cm}
\documentclass[smallcondensed]{svjour3} 
\smartqed 

\usepackage{float}
\floatstyle{ruled}
\newfloat{model}{thj_1}{lop}
\floatname{model}{Model}

\usepackage{graphics,enumitem} 
\usepackage{float}
\usepackage{graphicx} 
\usepackage{multirow}
 \usepackage{capt-of}
\usepackage{amsfonts}
\usepackage{bm,mathtools,mathabx}
\usepackage{booktabs,color}
\usepackage{tikz}
\usetikzlibrary{shapes.geometric,arrows}
\usepackage{empheq,lscape}
\usepackage{thesis}
\usepackage[ruled,lined]{algorithm2e}
\usepackage[bookmarks=false,colorlinks]{hyperref}
\hypersetup{
 linkcolor=blue, 
 citecolor=red,
 urlcolor=magenta
 }

\journalname{}

\begin{document}
\title{Exploiting sparsity for the min $k$-partition problem}
\author{Guanglei Wang \and Hassan Hijazi}

\institute{G. Wang \and H. Hijazi\at The Australian National University, 
 ACTON 2601, Canberra, Australia. \\
\email{guanglei.wang@anu.edu.au}
 \and
 H. Hijazi \at 
 Los Alamos National Laboratory,
 New Mexico, USA.\\
 \email{hlh@lanl.gov}
 }

\date{Received: date / Accepted: date}

\maketitle

\begin{abstract}
The minimum $k$-partition problem is a challenging combinatorial problem with a diverse set of applications ranging from telecommunications to sports scheduling. It generalizes the max-cut problem and has been extensively studied since the late sixties. Strong integer formulations proposed in the literature suffer from a prohibitive number of valid inequalities and integer variables. In this work, we introduce two compact integer linear and semidefinite reformulations that exploit the sparsity of the underlying graph and develop fundamental results leveraging the power of chordal decomposition. Numerical experiments show that the new formulations improve upon state-of-the-art.


\end{abstract}

\section{Introduction}
Given an undirected, weighted and connected graph $G=(V,E)$, the minimum $k$-partition (M$k$P) problem consists of partitioning $V$ into at most $k$ disjoint subsets, minimizing the total weight of the edges joining vertices in the same partitions. This problem has been shown to be strongly $\mathcal{NP}$-hard \cite{michael1979:computers}. M$k$P was firstly defined in~\cite{carlson1966:scheduling} in the context of scheduling where activities must be assigned to a limited number of facilities. The authors formulate the problem as a quadratic program and suggest a method returning local minima.
 
 
The complement of M$k$P is a max-$k$-cut problem where one is seeking to maximize the sum of weights corresponding to intra-partition edges. Related investigations can be found in~\cite{deza1990:complete,deza1992:clique,frieze1997:improved}. The well-known integer linear programming formulation of max-$k$-cut can be obtained by complementing the variables in the integer formulation of M$k$P. Thus solving max-$k$-cut is equivalent to solving M$k$P problem. 
{When $k = 2$, the minimum $2$-partition is equivalent to the max-cut problem, which is~$\mathcal{NP}$-hard and has been extensively studied in the literature (see e.g.,~\cite{barahona1986cut,goemans1995,deza1997geometry}).
Hence, in the general case, both M$k$P and max-$k$-cut problems are $\mathcal{NP}$-hard (see also~\cite{chopra1993:partition,papadimitriou1988}). 
However, they have different approximability results. It has been shown in~\cite{frieze1997:improved} that the max-$k$-cut can be solved $(1-k^{-1})$-approximately in polynomial time. 
But the M$k$P problem is not $\bigO(\abs{E})$-approximable in polynomial time unless $\mathcal{P= NP}$~\cite{eis2002:sdp}.}

{In~\cite{chopra1993:partition}, Chopra and Rao introduce two Integer Linear Programming (ILP) formulations for the M$k$P. The first one uses node and edge variables associated with the graph $G$.  
The second one uses edge variables but requires completing the original graph $G$ with missing edges.
To strengthen the formulations, several families of strong valid linear inequalities have been  proposed. 
We refer readers to~\cite{chopra1993:partition,chopra1995:mink,deza1990:complete,fairbrother2016:projection} for references.}

{\bf Contributions:} In this paper, we propose two compact reformulations of the
M$k$P problem. One is an ILP formulation based on Chopra and
Rao's~\cite{chopra1993:partition} edge-based model. The other is an ISDP
formulation based on the Eisenbl{\"a}tter's~\cite{eis2002:sdp} model. Both
models exploit the sparsity of the graph and thus have less integer variables
than their counterparts. 
{We also show that if $G$ is chordal, our proposed ILP formulation has $\abs{E}$ variables (according to \cite{chopra1993:partition}, no existing formulations live in that space)}. To validate the proposed formulations, several theoretical results are established. The computational experiments show that the proposed formulations are able to improve computational efficiency by several orders of magnitude. 

{\bf Outline:} The rest of the paper is organized as follows. In Section~\ref{sec:review}, the literature on the existing formulations of M$k$P is reviewed. 
{Applications of chordal graphs in compact reformulations are also discussed.} 
In Section~\ref{sec:theory}, we provide our compact formulations which are based on the maximal clique set. Some fundamental results are established. 
In Section~\ref{sec:test}, we describe some computational experiments and analyse the results. Finally, some concluding remarks are made in Section~\ref{sec:conclude}.

 {\bf Notation:} For a finite set $S$, $\abs{S}$ represents its cardinality. $\realS^n$ represents the set of real symmetric matrices in $\real^{n\times n}$. 
 $\realS^n_+$ denotes the cone of real positive semidefinite matrices in $\realS^n$, i.e., $\realS^n_+ = \bigcurly{S \in \realS^n: S \succeq 0}$. 
For a matrix $M$, $M_{C}$ represents the principle sub-matrix composed by set $C$ of columns and set $C$ of rows of $M$. We also use
$(C, D)$ to denote the matrix where $C$ and $D$ are concatenated by rows assuming
they have the same number of columns. The inner product between two matrices $A,B \in \real^{m×n}$ is denoted by $\scp{A}{B} = \sum\limits_{i =1}^m\sum\limits_{j=1}^n A_{ij} B_{ij}$.

\section{Literature Review} \label{sec:review}
In this section, we review some existing formulations of the M$k$P problem. 
To facilitate the development of our main results, we also review some basic graph terminology. 
\subsection{Formulations of M$k$P }
Following~\cite{chopra1993:partition,chopra1995:mink}, the M$k$P problem can be expressed in Model~\ref{model:kparties}.
\begin{model}[!h]
\caption{The edge formulation of the M$k$P problem}
\label{model:kparties}
\begin{subequations}
\begin{align}
\mbox{\bf variables:} \; \;& x \in \bigcurly{0, 1}^{\frac{n (n-1)}{2}} \label{kpart:binary}\\
\mbox{\bf minimize:} \;\; & \sum\limits_{\{i, j\} \in E} w_{ij}x_{ij} \label{kpart:obj}\\
 \mbox{\bf s.t.:} \;\;& x_{ih} + x_{hj} - x_{ij} \le 1, \; \forall i, j, h \in V, ~i<h<j\label{kpart:triangle} \\
 \;& \sum\limits_{i, j \in Q: i<j} x_{ij} \ge 1, \; \forall Q \subseteq V, \mathrm{where} ~\abs{Q} = k+1 \label{kpart:clique} 
\end{align}
\end{subequations}
\end{model}
Binary variable $x_{ij}$ is $1$ if and only if $i$ and $j$ are in the same partition, $0$ otherwise. 
The \emph{triangle} constraint~\eqref{kpart:triangle} enforces consistency with respect to partition membership. 
Namely, if $x_{ih}=1$ and $x_{jh}=1$, indicating that $i$, $h$ and $j$ are in
the same partition, this implies $x_{ij} = 1$. For every subset of $k + 1$
vertices, the \emph{clique} constraint~\eqref{kpart:clique} forces at least two
vertices to be in the same partition. Together with
constraints~\eqref{kpart:triangle}, this implies that there are at most $k$
partitions. In total, there are $3{\abs{V} \choose 3}$ triangle inequalities and
$\abs{V} \choose {k+1} $ clique inequalities. We refer to this formulation as the edge formulation. 

When $3 \le k \le \abs{V}-1$, we denote by $P^k$ the set of feasible solutions of Model~\ref{model:kparties} as:
 \begin{align}\label{eq:polytope}
P^k = \bigcurly{x \in \bigcurly{0, 1}^{V^2}: x \; \textrm{satisfies}~ \eqref{kpart:triangle} \;\textrm{and} \;\eqref{kpart:clique}},
\end{align}
where $V^2 := \bigcurly{\{i, j\} \in V \times V: i < j}$ represents the edges of the complete graph induced by vertices $V$. It is known that the convex hull of $P^k$ is a polytope \cite{chopra1993:partition}. Moreover, it is fully-dimensional in the space spanned by the edge variables~\cite{chopra1995:mink,deza1997geometry}. 

In~\cite{chopra1993:partition}, Chopra and Rao give several valid and facet defining inequalities for $P_k$. These include~\emph{general clique, wheel} and~\emph{bicycle} inequalities. 
Among others, we remark that the triangle~\eqref{kpart:triangle} and clique inequalities~\eqref{kpart:clique} can make the continuous relaxation of Model~\ref{model:kparties} hard to deal with. 
In fact, it has been noticed that the exact separation of the clique inequalities is $\mathcal{NP}$-hard in general, and the complete enumeration is intractable even for small values of $k$~\cite{ghaddar2011:branch}.
%

Chopra and Rao~\cite{chopra1993:partition} also propose an alternative ILP formulation, presented in Model~\ref{model:node-edge}. It has $kn + m $ variables, where $n$ and $m$ are respective numbers of nodes and edges of graph $G$. Binary variable $x_{ic}$ is $1$ if node $i$ lies in the $c$-th subset. Binary variable $y_{ij}$ is $1$ if $i$ and $j$ are in the same subset. {Some subsets can be empty.}
Constraints~\eqref{node-edge:assign} express that each node must be assigned to exactly one subset. 
Constraints~\eqref{node-edge:cut} indicates that if $\{i, j\}$ is
an edge and $i$ and $j$ are in the same subset then the edge $\{i, j\}$ is not cut by the partition (i.e., $y_{ij} = 1$). 
 Conversely, if nodes $i, j$ are in two disjoint subsets,~\eqref{node-edge:cut2} and~\eqref{node-edge:cut3} enforce that the edge $\{i, j\}$ is cut by the partition (i.e., $y_{ij} = 0$). 
We refer to Model~\ref{model:node-edge} as the node-edge formulation. 
Some valid inequalities have been proposed in the literature to strengthen Model~\ref{model:node-edge} (see, e.g., ~\cite{chopra1993:partition,fairbrother2016:projection}). 
\begin{model}[!h]
\caption{The node-edge formulation of the M$k$P problem}
\label{model:node-edge}
\begin{subequations}
\begin{align}
 \mbox{\bf variables:} \; \;& x \in \bigcurly{0, 1}^{kn}, y \in \bigcurly{0, 1}^{m} \label{node-edge:binary}\\
\mbox{\bf minimize:} \;\; & \sum\limits_{\{i, j\} \in E} w_{ij}y_{ij} \label{node-edge:obj}\\
 \mbox{\bf s.t.:} \;\;& \sum\limits_{c=1}^k x_{ic} = 1 \; i \in V \label{node-edge:assign} \\
 \;\;& x_{ic} + x_{jc} - y_{ij} \le 1 \;\; \{i, j\} \in E, c = 1, \dots, k \label{node-edge:cut}\\
 \;\;& x_{ic} - x_{jc}+ y_{ij} \le 1\; \; \{i, j\} \in E, c = 1, \dots, k
 \label{node-edge:cut2}\\
 \;\;&- x_{ic} + x_{jc} + y_{ij}\le 1 \; \; \{i, j\} \in E, c = 1, \dots, k.
 \label{node-edge:cut3}
\end{align}
\end{subequations}
\end{model}
%

Eisenbl{\"a}tter~\cite{eis2002:sdp} provides an exact ISDP reformulation
of Model~\ref{model:kparties}. We introduce variables $X_{ij} \in
\bigcurly{\frac{-1}{k-1}, 1}$, where $X_{ij} =1$ if and only if node $i$ and $j$
are in the same partition. This formulation is presented in Model~\ref{model:SDP1}.
\begin{model}[!h]
\caption{The integer SDP formulation of the M$k$P problem}
\label{model:SDP1}
\begin{subequations}
\begin{align}
 \mbox{\bf variables:} \; \;& X_{ij} \in \bigcurly{\frac{-1}{k-1}, 1} \;
 \;(\forall i, j \in V: i< j). \\
\min\;\; & \sum\limits_{\{i, j\} \in E} w_{ij} \frac{(k-1)X_{ij} +1}{k} \label{SDP1:obj}\\
 \subto \;\;& X \in \realS^n_+ \label{SDP1:1}\\
 		& X_{ii} = 1, \; \;\forall i \in V. \label{SDP1:2}
\end{align}
\end{subequations}
\end{model}

Replacing the constraint $X_{ij} \in \bigcurly{\frac{-1}{k-1}, 1}$ with
$\frac{-1}{k-1} \le X_{ij} \le 1$ yields a Semidefinite Programming (SDP) relaxation. Notice that $X_{ij} \le1$ can
be dropped since it is implicitly enforced by $X$ being positive
semidefinite and $X_{ii} = 1$ ({by considering the $2\times2$ minors}). This SDP relaxation was used in combination with
randomized rounding to obtain a polynomial time approximation algorithm for the
max $k$-cut problem by Freize and Jerrum~\cite{frieze1997:improved}. 
As proposed by~Eisenbl{\"a}tter~\cite{eis2002:sdp}, one can recast Model~\ref{model:SDP1} using binary variables. Indeed, we introduce binary variables $Y_{ij}$ such that 
\begin{align}\label{eq:mapping}
X_{ij} = \frac{-1}{k-1} + \frac{k}{k-1} Y_{ij}. 
\end{align}
Thus, $X_{ij} = \frac{-1}{k-1}$ if $Y_{ij} = 0$ and $X_{ij} = 1$ if $Y_{ij}=1$. This leads to 
 \begin{subequations}\label{model:SDP2}
 \begin{align}
\min\;\; & \sum\limits_{\{i, j\} \in E} w_{ij} Y_{ij} \notag \\
 \subto \;\;& \frac{-1}{k-1}J + \frac{k}{k-1} Y \in \realS^n_+, \label{SDP2:1}\\
 		& Y_{ii} = 1, \; \; \; \;\forall i \in V, \label{SDP2:2}\\
		& Y_{ij} \in \bigcurly{0, 1} \; \;\forall i, j \in V: i< j. \label{SDP2:3}
\end{align}
\end{subequations}
where $J$ is the all-one matrix. Note that the mapping function~\eqref{eq:mapping} is bijective or one-to-one. This bijective relationship will be used in Section~\ref{sec:extension}.

Several SDP based branch-and-bound frameworks for M$k$P problem (or max-$k$-cut) have appeared in the literature~\cite{ghaddar2011:branch,de2016:valid}. Typically, these approaches identify several families of valid inequalities to tighten the SDP relaxation during the solution procedure. However, solving large scale SDP problems is less efficient than their LP counterparts~\cite{wang2016:thesis,wang2017:envelope}. 

\subsection{Graph terminology}
We shall assume familiarity with basic definitions from graph theory. We will also use results from the algorithmic graph theory for chordal graphs and we refer readers to~\cite{Golumbic1980:chordal,blair1993:chordal} for relevant references. We introduce some of the graph notation and terminology that will be used in this article. 

Let $G=(V, E)$ be an~\emph{undirected graph} formed by vertex set $V$ and edge set $V$. The number of vertices is denoted by $n = \abs{V}$ and the number of edges by $m = \abs{E}$.
A graph is called \emph{complete} if every pair of vertices are adjacent. A \emph{clique} of a graph is an induced subgraph which is complete, and a clique is \emph{maximal} if its vertices do not constitute a proper subset of another clique. {A sequence of nodes $\{v_0, v_1, v_2, \dots, v_k, v_0\}\subseteq V$ is a \emph{cycle} of length $k+1$ if $\{v_{i-1}, v_i\} \in E$ for all $i \in \{1, \dots, k\}$ and $\{v_k, v_0\} \in E$.}

A graph is said to be \emph{chordal} if every cycle of length greater than or equal to $4$ has a \emph{chord} (an edge joining two nonconsecutive vertices of the cycle). 
Given a graph $G=(V, E)$, we say that a graph $G_F = (V, F)$ is a \emph{chordal extension} of $G$ if $G_F$ is chordal and $E \subseteq F$. Throughout this article, we also assume that the graph $G$ is connected. 
\subsection{{Chordal graphs and compact reformulations}}
{Investigations on chordal graph were initiated by~\cite{fulkerson1965:chordal,berge1966} for the characterisation of perfect graphs. 
Chordal graphs have a wide range of applications in combinatorial optimization~\cite{gavril1972,arnborg1989}, matrix completion~\cite{grone1984:psd,laurent1998:matrix} 
and more recently in sparse semidefinite programming~\cite{fukuda2001:exploit,nakata2003,waki2006:sos,kim2011:sparsity}. 
Comprehensive reviews on the fundamental theory and applications are given by Blair and Peyton~\cite{blair1993:chordal} and Vandenberghe and Andersen~\cite{vandenberghe2015}. 
%
}

{To the best of our knowledge, research on mathematical reformulations using chordal graphs started with Fulkerson and Gross~\cite{fulkerson1965:chordal}, where
the authors establish connections between interval graphs (a special case of chordal graph) and matrix total unimodularity. Later, the authors in~\cite{fukuda2001:exploit,nakata2003} 
introduce chordal graph techniques to accelerate the solution procedure of Interior Point Methods (IPMs) for SDPs by reformulating the underlying large matrix with an equivalent set of smaller matrices.  The reformulation is based on a clique decomposition of a chordal extension of the \emph{aggregated sparsity pattern graph}. 
This approach accelerates the computation of the solution (as a search direction in IPMs)  of the Schur complement equation, thereby speeding up
the solution procedure of IPMs. The reformulation method is often called conversion method or clique decomposition method. 
Chordal graphs are also used to exploit the \emph{correlative sparsity pattern} in the context of polynomial optimization
for deriving hierarchies of SDP relaxations (see~\cite{waki2006:sos,lasserre2006}). }

{More recently, Bienstock and Ozbay~\cite{Bienstock2004:tree} establish a connection between the well-known Sherali-Adam reformulation operator for 0-1 integer programs and chordal extensions. Later, the authors in~\cite{bienstock2015} show that polynomial-size LP reformulations can be constructed for certain 
classes of mixed-integer polynomial optimization problems by exploiting structured sparsity. }

{To distinguish and relate the results developed in this paper from the mentioned techniques, we remark the following:
 (1) our main results are ILP (Model~\ref{model:reduced}) and ISDP (Model~\ref{model:SDP1reduce}) reformulations, which is different from LP or SDP reformulations discussed above;
 (2) in addition to the known choral graph properties in the literature, our reformulations rely on the proposed fundamental results (Theorems~\ref{thm:proj} and~\ref{prop:sdpreduce});
 (3) to the best of our knowledge, the fundamental results  are novel in the sense that they are neither mentioned nor implied by any results in the literature. 
 Despite these differences, it is interesting to note that our reformulation is also related to the conversion method proposed in~\cite{fukuda2001:exploit,nakata2003}.  
 As will be shown later, both reformulations rely on a clique decomposition of the underlying graph, though with different purposes.}


\section{Main results} \label{sec:theory}

In this section, we propose two compact reformulations of the M$k$P problem by exploiting the structured sparsity of the underlying graph.
Let $\K$ be the set of all maximal cliques in $G_F$, representing the chordal extension of $G$, i.e., $\K =\bigcurly{C_1, \dots, C_l},~ C_i \subseteq V ~\forall i\in\{1,\dots,l\},$ such that 
$F = \bigcup_{r=1}^l C_r \times C_r \supseteq E$.
We will next show that the following property holds:
\begin{property}[Completion Property]
\label{prop:completion} 
For any vector $x_F \in \real^{\abs{F}}_+$, if $x_F$ satisfies 
\begin{subequations}
\begin{align}
& x_{ih} + x_{hj} - x_{ij} \le 1, \; \forall i, j, h \in C_r, r =1, \dots, l, \label{eq:maxclique1} \\
 & \sum\limits_{i, j \in Q} x_{ij} \ge 1, \; \forall Q \subseteq C_r, \mathrm{where} ~\abs{Q} = k+1, r = 1, \dots, l, \label{eq:maxclique2}\\
 & x_F \in \bigcurly{0, 1}^{\abs{F}}, \label{eq:maxclique3}
\end{align}
\end{subequations}
\end{property}
then there exists $x_T \in \bigcurly{0, 1}^{\abs{V^2 \setminus F}}$ such that $x= (x_F, x_T)$ satisfies \eqref{kpart:triangle} and \eqref{kpart:clique}. 

Under the conditions above, we observe that the value of the objective~\eqref{kpart:obj} can be determined by values of entries $x_F$ and independently of values in $x_T$. 
In addition, the resulting formulation has $\sum\limits_{r=1:\abs{C_r} \ge 3}^l 3{\abs{C_r} \choose 3} + \sum\limits_{r=1: \abs{C_r} \ge k+1}^l {\abs{C_r} \choose k+1}$ constraints and ${\abs {F}}$ binary variables. Thus it becomes more compact if this number is less than $3{\abs{V} \choose 3} + {\abs{V} \choose k+1} $. 



 \subsection{A compact ILP reformulation}
 Let $P^k_{F}$ represent the set of feasible solutions satisfying~\eqref{eq:maxclique1}, \eqref{eq:maxclique2}, \eqref{eq:maxclique3}, i.e., 
 \begin{align} \label{eq:feasibleset}
 P^k_{F} = \bigcurly{x \in \real^F_+: \eqref{eq:maxclique1}, \eqref{eq:maxclique2}, \eqref{eq:maxclique3}}. 
 \end{align}
We first show that the $\abs{F}$ components corresponding to index set $F$ in $x \in P^k$ represent a member of $P^k_F$. We denote by $\proj_{F} S$ the projection of $S$ into the space defined by the components in $F$, i.e., for a given set $S\subset \real^{I}, F \subset I \subseteq V^2$,
\begin{align}\label{operation:proj}
\proj_{F} S = \bigcurly{x_F \in \real^F: \exists x_T \in \real^{I\setminus F}, (x_F, x_T) \in S}
\end{align}
\begin{lemma}\label{lemma:easy}
 $\forall k,~2\le k \le n$, $\proj_{F} P^k \subseteq P^k_F$
\end{lemma}
\begin{proof}
Let $\bar x$ be an arbitrary point in $P^k$. It is easy to verify that entries of $\bar x_F$ satisfy constraints in $P^k_{F}$ and therefore $\proj_{F} P^k \subseteq P^k_{F} $. \qed
\end{proof}
\begin{lemma}[Lemma 3,~\cite{grone1984:psd}] \label{lemma:clique}
Given that $G_F$ is chordal, for any pair of vertices $u$ and $v$ with $u \ne v, \bigcurly{u, v} \not \in E$, the graph $G_F + \bigcurly{u, v}$ has a unique maximal clique which contains both $u$ and $v$. 
\end{lemma}
\begin{lemma} [Lemma 4,~\cite{grone1984:psd}] \label{lemma:grone}
Given that $G_F$ is chordal, there exists a sequence of chordal graphs 
\begin{align}
& G_i = (V, F_i), \; i = 0, \dots, s,
\end{align}
such that $G_0= G_F$, $G_s$ is the complete graph, and $G_i$ is obtained by adding an edge to $G_{i-1}$ for all $i = 1,\dots, s$. 
\end{lemma}
For sake of further development, let us represent the neighborhood of a vertex $u \in V$ as
$N_G(u) = \bigcurly{v \in V: \{u, v\} \in E}$.
%
\begin{proposition}\label{prop:maximalclique}
Given a chordal graph $G=(V, E)$, for every edge $\{u, v\} \in E$, the unique maximal clique $C$ containing both $u$ and $v$ is the union of $\bigcurly{u, v}$ and $N_{G} (u) \bigcap N_{G}(v)$, i.e., 
$C = \bigcurly{u, v} \bigcup \left (N_{G} (u) \bigcap N_{G}(v) \right )$. 
\end{proposition}
\begin{proof}
If $N_{G} (u) \bigcap N_{G}(v) = \emptyset$, the result holds. When $N_{G} (u) \bigcap N_{G}(v)$ is nonempty, we observe that $N_{G} (u) \bigcap N_{G}(v)$ contains all cliques involving $\{u, v\}$, minus $\{u, v\}$. Indeed, if $L$ is a clique containing both $u$ and $v$, then any element (other than $u$ and $v$) in $C$ is a neighbor of $u$ and $v$, which means $L \subseteq N_{G} (u) \bigcap N_{G}(v)$. Hence if we can show that $C = \bigcurly{u, v} \bigcup \left (N_{G} (u) \bigcap N_{G}(v) \right )$ is a clique, then by definition, $C$ is the maximal clique containing $\bigcurly{u, v}$. Indeed, suppose there exist $l_1, l_2 \in N_G(u) \bigcap N_{G}(v) $ such that $l_1$ and $l_2$ are disconnected, then the set $\bigcurly{u, l_1, v, l_2, u}$ is a cycle of length 4 without a chord, contradicting with the fact that $G$ is chordal. The uniqueness comes from the maximality of $C$. \qed
\end{proof}

\begin{theorem} \label{thm:proj}
$\forall k,~2\le k \le n$, $P^k_{F}=\proj_{F} P^k$. 
\end{theorem}
\begin{proof}
By Lemma~\ref{lemma:easy}, $\proj_{F} P^k \subseteq P^k_{F}.$ We now show that the reverse also holds. 
 Given $G_F$, let $\{G_0, G_1, \dots G_s\}$ be a sequence of choral graphs satisfying Lemma~\ref{lemma:grone}.
 Denote by $\{i_1 ,j_1\}$ the edge of $G_1$ that is not a member of $G_0$ (w.l.o.g., we assume $i_1 < j_1$). Let $\bar x_F$ be a point in $P^k_{F}$. If we can show that there exists a vector $\bar x_1 = (\bar x_F,x_{i_1j_1})$ satisfying the constraints in $P^k_{G_1}$, then by induction, we can show the existence of a point $x_s \in P^k$.
 
By Lemma~\ref{lemma:clique}, there is a unique maximal clique $C$ of $G_1$ containing nodes $i_1, j_1$ and it can be identified by Propoistion~\ref{prop:maximalclique}.
Without loss of generality, we may reorder indices of the nodes in $C$ and let $C=\bigcurly{l_0, \dots,l_p,i_1,j_1}$ with $l_0 < \dots <l_p <i_1 <j_1$. Let $x_C$ be the vector corresponding to the maximal clique $C$. Since any clique containing $\{i_1,j_1\}$ is a subset of $C$, $x_1 \in P^k_{G_1}$ iff $x_C \in P^k_{C}$, where $P^k_{C}$ is formed by replacing $V$ with $C$ in~\eqref{eq:polytope}.
 
Hence we only need to show that $x_C \in P^k_{C}$ for some value of $x_{i_1j_1}$. Given $x_F$, we construct $x_{i_1j_1}$ as follows: if $C=\bigcurly{i_1, j_1}$, let $x_{i_1j_1} =1$ and the solution is feasible. Otherwise, three cases can occur: 
\begin{enumerate}
\item \label{allzeros} $(x_{hi_1}, x_{hj_1}) = (0, 0), \forall h \in \bigcurly{l_0,\dots, l_p}$. 
\item \label{existsone} $\exists h \in \bigcurly{l_0,\dots, l_p}$ such that $x_{hi_1} + x_{hj_1} = 1$.
\item \label{existstwo} $\exists h \in \bigcurly{l_0,\dots, l_p}$ such that $x_{hi_1} + x_{hj_1} = 2$.
\end{enumerate}
We now construct the solution as follows:
\begin{enumerate} 
\item If case~\ref{allzeros} or case~\ref{existstwo} occurs, let $x_{i_1j_1}=1$.
\item otherwise (i.e., case~\ref{existsone} occurs), let $x_{i_1j_1} = 0$.
\end{enumerate} 
In order to show that the constructed solution is {valid}, we next show that case~\ref{allzeros} and case~\ref{existstwo} are exclusive with respect to case~\ref{existsone}. 
It is obvious that case \ref{allzeros} is exclusive with case \ref{existsone}. We now show that case \ref{existsone} and case \ref{existstwo} are exclusive:
\begin{itemize}[label={$\bullet$}]
\item The result is straightforward when $p=0$. 
\item When $p \ge 1$, we proceed by contradiction: suppose there exists $h, h' \in V: h \neq h'$~ such that $x_{hi_1} + x_{hj_1} = 1$ and $x_{h'i_1} + x_{h'j_1} = 2$. Let us consider $(x_{hi_1}, x_{hj_1}) = (0, 1)$. The other case $(x_{hi_1}, x_{hj_1}) = (1, 0)$ will follow symmetrically. Note that $\{h,h', i_1\}$ and $\{h, h', j_1\}$ are cliques of length 3 in $G_F$. Thus by~\eqref{eq:maxclique1}, we have, on one hand, $x_{hh'} + x_{h'i_1} - x_{hi_1} \le 1$ leading to $x_{hh'} \le 0$. On the other hand, we have $ x_{h'j_1} + x_{hj_1} - x_{hh'} \le 1$ leading to $x_{hh'} \ge 1$, contradiction.
\end{itemize}
We now verify that the extended solution $x_1= (x_F, x_{i_1j_1})$ satisfies constraints in $P^k_C$. Since $\bigcurly{l_0, \dots,l_p,i_1}$ and $\bigcurly{l_0, \dots,l_p,j_1}$ are cliques in $G_F$, the associated constraints have been imposed by $P^k_F$. Thus, we just need to verify that $x_1$ satisfies 
 \begin{subequations}
\begin{align}
&x_{hi_1} + x_{hj_1} - x_{i_1j_1} \le 1, \forall h \in \bigcurly{l_0, \dots, l_p} \label{eq:tri1} \\
&x_{hi_1} + x_{i_1j_1} -x_{hj_1} \le 1, \forall h \in \bigcurly{l_0, \dots, l_p} \label{eq:tri2}\\
&x_{hj_1} + x_{i_1j_1} - x_{hi_1} \le 1, \forall h \in \bigcurly{l_0, \dots, l_p}\label{eq:tri3}\\
&\sum\limits_{\{h,f\} \in C_q} x_{hf} + x_{i_1j_1} \ge 1, \; \forall Q\subseteq C, \abs{Q} = k+1,~i_1, j_1 \in Q \label{eq:cli}
\end{align}
\end{subequations}
where $C_q = \bigcurly{\{h,f\} \in Q^2:~ h < f,~\{h, f\} \ne \{i_1, j_1\}}$. 

\begin{itemize}[label={$\bullet$}]
\item First, let us show that the the constructed solution satisfies the above triangle inequalities~\eqref{eq:tri1}--\eqref{eq:tri3}. For case~\ref{allzeros}, for each $h, i_1, j_1\in C $, the unique solution $(x_{hx_1}, x_{hj_1}, x_{i_1j_1}) = (0, 0, 0)$ is feasible. For case~\ref{existsone}, both solution $(x_{hi_1}, x_{hj_1}, x_{i_1j_1}) = (1, 0, 0)$ and $(x_{hi_1}, x_{hj_1}, x_{i_1j_1}) = (0, 1, 0)$ are feasible. For case~\ref{existstwo}, the unique solution $(x_{hi_1}, x_{hj_1}, x_{i_1j_1})= (1, 1, 1)$ is feasible. 

\item Second, we need to verify that the constructed solution is feasible for the clique inequalities~\eqref{eq:cli} when $\abs{C} \ge k+1$.
It is easy to see that this is true for case~\ref{allzeros} and case~\ref{existstwo}. For case~\ref{existsone}, we consider $\abs{C} = k+1$ and $C \ge \abs{k+2}$. Let us denote by $h^*$ the index such that $x_{h^*i_1} + x_{h^*j_1} = 1$. Recall that $x_{i_1j_1} = 0$. 
\begin{enumerate}
\item When $\abs{C}=k+1$, it holds that $Q= C$, $\{h^*, i_1\} \in C_q$ and $\{h^*, j_1\} \in C_q$. Then we have 
\begin{align*}
&\sum\limits_{\{h,f\} \in C_q} x_{hf} + x_{i_1j_1} \ge 1,
\end{align*}

\item When $\abs{C} \ge k+2$, there are multiple subsets $Q$. If $Q$ contains $h^*$, the desired result follows. If $h^*\not \in Q$, we have 
\begin{align*}
&\sum\limits_{\{h,f\} \in C_q} x_{hf} + x_{i_1j_1}& = \sum\limits_{\{h, f \} \in C_q: f < i_1} x_{hf} +\sum\limits_{\{h, i_1\} \in C_q} x_{hi_1}+\sum\limits_{\{h, j_1\} \in C_q} x_{hj_1}
\end{align*}
Note that if the following disjunction is true:
$$\sum\limits_{\{h,f\} \in C_q: f < i_1} x_{hf}\ge 1 \vee \sum\limits_{\{h,i_1\} \in C_q} x_{hi_1} \ge 1 ~\vee~\sum\limits_{\{h, j_1\} \in C_q} x_{hj_1} \ge 1$$ 
then we have $\sum\limits_{\{h, f\} \in C_q} x_{hf} + x_{i_1j_1} \ge 1$, the result follows. We now show that it is not possible to have a solution $x_F \in P^k_F$ satisfying 
\begin{subequations}
\begin{align}
&\sum\limits_{\{ h,f\} \in C_q: f < i_1} x_{hf} = 0 \label{eq:sub1}\\
&\sum\limits_{\{h,i_1\} \in C_q} x_{hi_1} = 0 \label{eq:sub2}\\
&\sum\limits_{\{h,j_1\} \in C_q} x_{hj_1} = 0. \label{eq:sub3}
\end{align}
\end{subequations}
We proceed by contradiction. Assume that \eqref{eq:sub1}-\eqref{eq:sub3} hold. Then we have:
\begin{align*}
0=&\sum\limits_{\{h, f\} \in C_q: f < i_1} x_{hf} +\sum\limits_{\{h, i_1\} \in C_q} x_{hi_1}+\sum\limits_{\{h, j_1\} \in C_q} x_{hj_1} \\
=& \sum\limits_{\{h, i_1\} \in C_q} x_{hi_1}+\sum\limits_{\{h, j_1\} \in C_q} x_{hj_1}\\
\ge & 2 - 2\sum\limits_{\{h, f\} \in C_q: f< i_1} x_{hf} - 2\sum\limits_{\{h, h^*\} \in C_q} x_{hh^*} - x_{h^*{i_1}} - x_{h^*{j_1}} \\
\ge & 1 - 2\sum\limits_{\{h, h^*\} \in C_q} x_{hh^*}\\
= & 1, 
\end{align*}
which is a contradiction. 
The first``$\ge$" comes from the fact that $Q \setminus \bigcurly{i_1} \cup \{h^*\}$ and $Q \setminus \bigcurly{j_1} \cup \{h^*\}$ are cliques of size $k+1$. Thus by clique inequalities~\eqref{eq:maxclique2} enforced in $P^k_F$, we have 
\begin{align*}
&\sum\limits_{\{h, f\} \in C_q: f< i_1} x_{hf} +\sum\limits_{\{h, j_1\} \in C_q} x_{hj_1} + \sum\limits_{\{h, h^*\} \in C_q} x_{hh^*}+ x_{h^*{j_1}} \ge 1, \\
&\sum\limits_{\{h, f\} \in C_q: f< i_1} x_{hf} +\sum\limits_{\{h, i_1\} \in C_q} x_{hi_1} + \sum\limits_{\{h, h^*\} \in C_q} x_{hh^*}+ x_{h^*{i_1}} \ge 1.
\end{align*}
The second ``$\ge$" comes from~\eqref{eq:sub1} and $ x_{h^*{i_1}} + x_{h^*{j_1}} = 1.$ 
The last equality comes from the fact that for each $h\in Q: h < i_1, h\neq h^*$, $\{h, h^*, i_1\}$ and $\{h, h^*, j_1\}$ are cliques in~$G_F$. 
Thus, by \eqref{eq:sub2}-\eqref{eq:sub3} and the triangle inequality~\eqref{eq:maxclique1} in $P^k_F$, we have
\begin{align*}
x_{hh^*} \le 1+ x_{hi_1} - x_{h^*i_1} \; \; \textrm{and} \; \; x_{hh^*} \le 1+ x_{hj_1} - x_{h^*j_1}, \forall h\in Q: h < i_1. 
\end{align*}
which leads to $x_{hh^*} = 0, \forall h\in Q: h < i_1, h\neq h^* $.
\end{enumerate}
\end{itemize}
\qed
\end{proof}
 This immediately yields the following desired result. 
 \begin{corollary}\label{cor:kpart}
 Model~\ref{model:kparties} is equivalent to Model \ref{model:reduced} in the sense that for any feasible solution for Model $\ref{model:kparties}$, there exists a feasible solution for~\ref{model:reduced} such that their objective values are equal; conversely, for any feasible solution $x_F$ for Model $\ref{model:reduced}$, there also exists a vector extending $x_F$ such that it is feasible for Model~\ref{model:kparties}. 
 \begin{model}[!h]
\caption{The clique-based reformulation of the M$k$P problem}
\label{model:reduced}
%
 \begin{align*}
 \mbox{\bf variables:} \; \;& x \in \bigcurly{0, 1}^{\abs{F} } \\
\mbox{\bf minimize:} \;\; & \eqref{kpart:obj}\\
 \mbox{\bf s.t.:} \;\;&\eqref{eq:maxclique1}, \eqref{eq:maxclique2}.
\end{align*}
\end{model}
 \end{corollary}
 \begin{proof}
Observe that the objective~\eqref{kpart:obj} is determined by variables $x_{ij}, \forall \{i, j\}\in E$. Given that $E \subseteq F$ and based on Theorem~\ref{thm:proj}, the result follows. \qed
 \end{proof}
%
 Observe that Model~\ref{model:reduced} has less binary variables than Model~\ref{model:kparties} when $G$ is not complete. In the most favorable case, where the given graph $G=(V, E)$ is chordal, Model~\ref{model:reduced} involves no additional binary variables and thus becomes more attractive than Model~\ref{model:node-edge}. 
{In Section 2 of~\cite{chopra1993:partition}, the authors state that if $G$ is not complete, they are not aware of any formulation that uses only edge variables.}
Here, we see from Corollary~\ref{cor:kpart} that if $G$ is chordal, Model~\ref{model:reduced} uses exactly $\abs{E}$ variables. 
 
\subsection{A compact ISDP reformulation}\label{sec:extension}
As remarked before, the compact reformulation~\ref{model:reduced} of Model~\ref{model:kparties} exploits the structured sparsity of the chordal extension of the original graph $G$. Similarly, we show in this section that there exists a clique-based reformulation of Model~\ref{model:SDP1}. It is presented in Model~\ref{model:SDP1reduce}.
\begin{model}[!h]
\caption{The clique-based integer SDP formulation of the M$k$P problem}
\label{model:SDP1reduce}
\begin{subequations}
\begin{align}
 \mbox{\bf variables:} \; \;& X_{ij} \in \bigcurly{\frac{-1}{k-1}, 1} \; \;(\forall \{i, j\}\in F) \\
\min\;\; & \eqref{SDP1:obj} \notag \\
 \subto \;\;& X_{ii} = 1, \; \;\forall i \in V. \label{SDP1reduce:1} \\
 & X_{C_r} \in \realS^n_+, \; \forall C_r \in \K \label{SDP1reduce:2}.
\end{align}
\end{subequations}
\end{model}

To show that Model~\ref{model:SDP1reduce} is a valid formulation of M$k$P problem, it is sufficient to show that 
\begin{align*}
\proj_{F} \bigcurly{x \in \bigcurly{\frac{-1}{k-1}, 1}^{n(n-1)/2}:\eqref{SDP1:1},\eqref{SDP1:2} } = \bigcurly{x \in \bigcurly{\frac{-1}{k-1}, 1}^F: \eqref{SDP1reduce:1},\eqref{SDP1reduce:2}}, 
\end{align*}
where the projection operator $\proj_{F} S$ has been defined in~\eqref{operation:proj}.
With the bijective mapping defined in~\eqref{eq:mapping}, we see that the left-hand side and the right-hand side of the above equation correspond to the respective $\F$ and $\F'$ presented below, 
 \begin{align*}
 &\F = \bigcurly{y \in \bigcurly{0, 1}^{\frac{n(n-1)}{2}}: \; \eqref{SDP2:1}, \eqref{SDP2:2}}, \\
 &\F' = \bigcurly{y \in \bigcurly{0, 1}^{F}: \frac{-1}{k-1}J_{C_r} + \frac{k}{k-1} Y_{C_r} \in \realS^n_+, \; \forall C_r \in \K, \; Y_{ii} = 1, \; \forall i \in V}.
\end{align*}

Thus we just need to prove that $\proj_{F}\F= \F'$. 
To this end, we exploit a technical lemma that was proposed in~\cite{frieze1997:improved}. A similar lemma has been used by Eisenbl{\"a}tter's~\cite{eis2002:sdp} to prove the equivalence of Model~\ref{model:SDP1} and Model~\ref{model:kparties}. 

\begin{lemma} [Lemma 1, \cite{frieze1997:improved}] \label{lemma:frieze}
For all integers $n$ and $k$ satisfying $2 \le k \le n + 1$, there exist $k$ unit vectors $\bigcurly{u_1, \dots, u_k} \in \real^n$ such that $\scp{u_l}{u_h} = \frac{-1}{k-1},$ for $l \ne h$.
\end{lemma}

\begin{theorem}\label{prop:sdpreduce}
Given integer $2 \le k \le n, \proj_{F}\F= \F'$.
\end{theorem}
\begin{proof}
As remarked in~Section~\ref{sec:review}, $\F= P^k$ defined in~\eqref{eq:polytope}. 
Thus it holds that $\proj_{F}\F= \proj_{F} P^k$. By Theorem~\ref{thm:proj}, $\proj_{F} P^k = P^k_F$. 
Thus, it is sufficient to show that $P^k_F = \F'$. 
Although the proof is close to the reasoning of proving $P^k = \F$ given by Eisenbl{\"a}tter's~\cite{eis2002:sdp}, 
we state it below for the sake of completeness.

\begin{itemize}[label={$\bullet$}]
\item $ \F' \subseteq P^k_F$: Let $\bar y $ be any binary vector in $\F'$. We show that $\bar y$ satisfies the triangle inequalities~\eqref{eq:maxclique1} and clique inequalities~\eqref{eq:maxclique2}. 
\begin{enumerate}
\item Triangle inequalities~\eqref{eq:maxclique1}: Suppose there exists $i, j, h \in \C_r$ such that $(y_{ij}, y_{ih}, y_{jh})$ violates~\eqref{eq:maxclique1}. 
Then, $\{y_{ij}, y_{ih}, y_{jh}\}$ can only be assigned the values $(1, 1, 0)$, $(1, 0, 1)$ or $(0, 1, 1)$. 
Since $\bar y \in \F'$, the principle sub-matrix of $X_{C_r} = \frac{-1}{k-1}J_{C_r} + \frac{k}{k-1} Y_{C_r}$ corresponding to indices $(i, j, h)$ in the following form
\begin{align} 
\begin{pmatrix}
1 & x_{ij} & x_{ih} \\
x_{ij} & 1 & x_{jh} \\
 x_{ih} & x_{jh} & 1
\end{pmatrix}
\end{align}
is positive semidefinite. 
This implies its determinant $1+2x_{ij}x_{ih}x_{jh}- x_{ij}^2 -x_{ih}^2 - x_{ij}^2 \ge 0 $. 
One can verify that if $\{y_{ij}, y_{ih}, y_{jh}\}$ is assigned value $(1, 1, 0)$, $(1, 0, 1)$ or $(0, 1, 1)$, the determinant becomes $-(\frac{k}{k-1})^2 < 0$, contradiction. 
\item Clique inequalities~\eqref{eq:maxclique2}: Suppose there exists $Q\subseteq \C_r: \abs{Q} = k+1$ such that $ y_Q$ violates the clique constraint~\eqref{eq:maxclique2}. This implies that $x_{ij} = \frac{-1}{k-1}, ~\forall (i,j) \in Q$, leading to $\sum\limits_{i, j \in Q: i<j} x_{ij }= \frac{-k(k+1)}{2(k-1)}$.
Note that $X_{Q} = \frac{-1}{k-1}J_{Q} + \frac{k}{k-1} Y_{Q} \in \realS^{\abs{k+1}}_{+}$,  suggesting that 
$\allones^T X_Q \allones= \sum\limits_{i,j \in Q} 2x_{ij} + k+1 \ge 0$, where $\allones$ is the all ones vector.
 This leads to $\sum\limits_{ij} x_{ij }\ge \frac{-(k+1)}{2} > \frac{-(k+1)}{2}\times \frac{k}{k-1}$, contradiction. 
\end{enumerate}

\item $P^k_F \subseteq \F'$: for any $\bar y \in P^k_F$, we show that for each $C_r \in \K$ ($\abs{C_r} \ge 3$), the matrix $X_{C_r} = \frac{-1}{k-1}J_{C_r} + \frac{k}{k-1} Y_{C_r}$ formed by vector $y_{C_r}$ is positive semi-definite. Given $k$, let $\{u_1,\dots, u_k\}$ be a set of unit vectors satisfying Lemma~\ref{lemma:frieze}. Now we construct a real matrix $B\in \real^{n \times \abs{C_r}}$ in the following way. For each $i, j \in C_r$,
\begin{enumerate}
\item If $y_{ij} = 0$, then assign column $i, j$ of matrix $B$ with different unit vectors from set $\{u_1,\dots, u_k\}$.
\item Otherwise, assign column $i, j$ with the same unit vector from set $\{u_1,\dots, u_k\}$.
\end{enumerate}
One can now verify that matrix $X_{C_r} = B^TB$, showing that $X_{C_r}$ is positive semidefinite. 
\end{itemize}
\qed
\end{proof}
In summary, the relationship between the aforementioned four constraint sets,
i.e., $P^k$, $\F$, $P^k_F$ and $\F'$, can be depicted in
Figure~\ref{fig:relation}, where the dotted arrow represents the projection operator $\proj_{F}$.
\begin{figure}
\begin{center}
\begin{tikzpicture}
 \draw (0.3, 0) node (origin) {$P^k$};
 \draw (1.7, 0) node (sdp){$\F$};
 \draw (0.3, -1.5) node (reduce) {$P^k_F$};
 \draw (1.7, -1.5) node (sdpreduce){$\F'$};
 \draw (1, 0) node {$\eqdef$};
 \draw (1, -1.5) node {$\eqdef$};
 \draw [->, dotted] (origin) -- (reduce);
 \draw [->, dotted] (sdp)-- (sdpreduce);
\end{tikzpicture}
\caption{The relationship between the four constraint sets}
\label{fig:relation}
\end{center}
\end{figure}
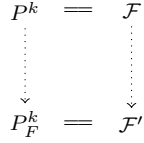

Note that the number of integer variables in Model~\ref{model:SDP1reduce} is generally smaller than that of Model~\ref{model:SDP1}, making it attractive. The number of constraints in Model~\ref{model:SDP1reduce} grows linearly with respect to the size of the clique set $\K$. The value $\abs{\K}$, as will be shown in Section~\ref{sec:construction}, is bounded by $(n -2)$. 

Analogously, four continuous relaxations of the four constraint sets can be formed by relaxing the respective integrality constraints. We denote by
$\overline{P^k}, \overline{\F}, \overline{P^k_F}$ and $\overline{\F'}$ the
respective continuous relaxation sets of $P^k, \F, P^k_F$ and $\F'$. 
It has been shown in~\cite{eis2002:sdp} that $\overline{P^k}$
and $\overline{P^k_{sdp}} $ are not contained in one another. We remark that
similar results also hold for $\overline{P^k_F}$ and $\overline{\F'}$. This will
be illustrated numerically in the subsequent section.

\subsection{The construction of the clique set}\label{sec:construction}
Model~\ref{model:reduced} and Model~\ref{model:SDP1reduce} are attractive when
they involve less constraints and integer variables than the respective
Models~\ref{model:kparties} and~\ref{model:SDP1}. Hence one would like to find an
``optimal'' (maximal) clique set $\K$ that minimizes the number of constraints and variables in these models.
{Similar to~\cite{fukuda2001:exploit,nakata2003}, we remark that it is hard to determine such a chordal extension $G_F$ and consequently the clique set $\K$. 
Alternatively, one may want to find a clique set such that the number of edges of $G_F$ is minimized, which however is $\mathcal{NP}$-complete~\cite{yannakakis1981}.}

Nevertheless, for many practical purposes, there are good methods to find chordal extensions such that the size of $\K$ (measured by $\max\limits_{C_r \in \K} \abs{C_r}$) is small. We refer readers to~\cite{bodlaender2010:treewidth} for an excellent survey. Here, we employ the~{\tt greedy fill-in heuristic}~\cite{koster2001:treewidth} to obtain $G_F$ and $\K$. The greedy fill-in heuristic attempts to create few new edges, which leads to less integer variables in Models~\ref{model:reduced} and~\ref{model:SDP1reduce}.
To make the text self-contained, we present the algorithm below. 
\begin{algorithm}[!ht]
\caption{Heuristic to find a chordal extension~\cite{koster2001:treewidth}} \label{alg:chordal}
\DontPrintSemicolon 
\SetKwInOut{Input}{input}
\SetKwInOut{Output}{output}
\Input{ Graph $G = (V, E)$}
\Output{A chordal graph $G_F= (V, F)$ of $G$ and cliques $K_i$}
Initialisation: $H = G$, $i=0$, $G_F = (V, F)$ with $F=E$ \;
\While{$\#$ Vertices of H $\ge 3$}
{
\eIf{all fill-in values of vertices in $H$ is $0$}
{ terminates \;} 
{Choose a vertex $v$ with the smallest number of fill-in edge \;
Label $v$ with $i$ (i.e., $v_i$) \;
Make the neighbouring vertices $N_H(v_i)$ of $v_i$ a clique \;
$K_i = N_H(v_i) \cup \{v_i\}$ \; 
$F = F \bigcup \bigcurly{(u, v_i): u \in N_H(v_i)}$ \;
$i = i+1$\; 
Remove $v_i$ from $H$\;
}
 }
\end{algorithm}
The term~\emph{fill-in} of a vertex refers to the number of pairs of its
non-adjacent neighbors. Note that we terminate the algorithm when the number of
nodes in $H$ is less than $3$ as all valid
inequalities~\eqref{eq:maxclique1}-\eqref{eq:maxclique2} are based on maximal
clique sets with size $\geq 3$. 

We now need to extract the maximal clique set $\K$ from $\bigcurly{K_i: i=1,\dots, n-2}$. 
Due to~\cite{fulkerson1965:chordal}, it is known that $\K$ contains exactly sets $K_i$ for which there exists no $K_i, K_j: i< j$, such that $K_j \subset K_i$. This also shows that the cardinality of $\K$ is bounded by $(n-2)$. It should also be noted that for each two distinct (maximal) cliques $(C_r, C_s) \in \K \times \K$ such that $\abs{C_r \bigcap C_s} \ge 3$, there are redundant triangle inequalities~\eqref{eq:maxclique1} in variables $x_{ij}, \{i, j\} \in C_r \bigcap C_s$. Similarly, for each two distinct (maximal) cliques $(C_r, C_s) \in \K \times \K$ such that $\abs{C_r \bigcap C_s} \ge k + 1$, there are redundant clique inequalities~\eqref{eq:maxclique2}. This redundancy can be avoided by checking the occurrence of each associated tuple. 
\subsection{The separation of valid inequalities} 
Model~\ref{model:reduced} can also suffer from a prohibitive number of clique inequalities when the size of some maximal clique set $C_r \in \K$ is large. This issue can be alleviated by a separation algorithm approach. 
For a maximal clique set $C_r \in \K$ with $\abs{C_r} \ge k+1$, the number of clique constraints is ${\abs{C_r}} \choose {k+1}$. This number grows roughly as fast as $\abs{C_r}^k$ as long as $2k \le \abs{C_r}$. 
We have noticed that some heuristics for the separation of cliques and other inequalities have been proposed in~\cite{eis2001:frequency,kaibel2011:orbit,de2016:valid,fairbrother2016:projection} and remark that these separation algorithms can be adapted for both Model~\ref{model:reduced} and~\ref{model:SDP1reduce} in a branch and bound algorithm to solve the M$k$P problem to global optimality. We
leave the sufficiently thorough investigation for future research.

\section{Numerical experiments}\label{sec:test}
Results in this section illustrate the following key points:
\begin{enumerate}
\item Model~\ref{model:reduced} is more scalable than Model~\ref{model:kparties} for general graphs. 
\item Compared with Model~\ref{model:node-edge}, Model~\ref{model:reduced} is
 quite attractive when the underlying graph is chordal. For general random
 graphs, it is less competitive for branch-and-bound although it provides stronger continuous relaxation bounds. This is mainly due to the exponential number of clique inequalities~\eqref{eq:maxclique2}. 
\item The continuous relaxation of Model~\ref{model:SDP1reduce} is computationally more scalable than that of Model~\ref{model:SDP1} when the underlying graph has structured sparsity. 
\item The continuous relaxation of Model~\ref{model:SDP1reduce} and that of~\ref{model:reduced} do not dominate each other. 
\end{enumerate}

\subsection{Test instances}
To illustrate the above key points, we randomly generate four sets of sparse graphs. The first set includes {\tt band} graph instances, which were used in~\cite{fukuda2001:exploit,nakata2003}.
The other sets are generated by the package {\tt rudy}~\cite{rinaldi1998:rudy}. Similar instances have been used in~\cite{ghaddar2011:branch,de2016:valid} for numerical demonstration.
\begin{itemize}[label={$\bullet$}]
\item {\tt band}: we generate graphs with edges set $E = \bigcurly{\{i, j\} \in
 V \times V: j - i \le \alpha, i < j}$, where $\alpha$ is 1 plus the
 partition parameter $k$, i.e., $\alpha = k +1$.
 The $50\%$ of edge weights are $-1$ and the others are $1$. 


\item {\tt springlass2g}: Eleven instances of a toroidal two dimensional grid with gaussian interactions. The graph has $n = $(rows $\times$ columns) vertices. 
\item {\tt spinglass2pm}: generates a toroidal two-dimensional grid with $\pm1$ weights. The grid has size $n =$ (rows $\times$ columns). The percentage of negative weights is $50 \%$.
 \item {\tt rndgraph}: we generate a random graph of $n$ nodes and density
 10\%. The edge weights are all $1$. 
\end{itemize}

\subsection{Implementation and experiments setup}
Models~\ref{model:kparties},~\ref{model:node-edge},~\ref{model:SDP1},~\ref{model:reduced} and~\ref{model:SDP1reduce}
are encoded in C++. MIP problem instances are solved by CPLEX 12.7~\cite{Cplex} with default settings. 
The SDP relaxation instances are solved by the state-of-the-art solver MOSEK 8~\cite{mosek} with default tolerance settings, which exploits the sparsity for performance and scalability gains. 

The experiments are conducted on a Mac with Inter Core i5 clocked at 2.7 GHz and
with 8 GB of RAM. For a fair computational comparison, CPU time is used for all
computations. A wall-clock time limit of 10 hours was used for all computations.
If no solution is available at solver termination or the solution process is
killed by the solver, (--) is reported. 

\subsection{Analysis of the computational results}
\subsubsection*{Results on Model~\ref{model:kparties} and Model~\ref{model:reduced}}
As remarked in previous sections, both Models~\ref{model:kparties} and~\ref{model:reduced} suffer from a prohibitive number of clique inequalities. Thus we fix $k=3$. 
For each problem instance, we measure the computational performance of each
formulation by computational time and optimality gap (if it has). 
All Computational time is measure by the CPU time in seconds. 
As all relaxations lead to a lower bound of the optimum, we quantify optimality gap as
\begin{align*}
 \textrm{gap} = \frac{\textrm{optimum - lower bound}}{ \textrm{optimum}} \times 100
\end{align*}

The numerical results are presented in Tables~\ref{tab:spinglass}. For each problem instance, the statistics on root node relaxation and the full branch-and-bound procedure are reported. These tested cases illustrate the following key points.
\begin{enumerate}
\item The compact Model~\ref{model:reduced} is remarkably more scalable for all tested instances than Model~\ref{model:kparties}. 
\item The continuous relaxations of both models are strong. It is also worth mentioning that the optimality gaps at root node for Model~\ref{model:reduced} are nearly the same with the original model~\ref{model:kparties}.
\item For problem instances with over 100 vertices, the computational time for Model~\ref{model:reduced} grows exponentially as the problem size increases. This is probably due to the fact that the size of each clique in $\K$ becomes larger, leading to an exponential number of clique inequalities~\eqref{eq:maxclique2}.
\end{enumerate}
\begin{table}[h!]
 \begin{center}
 	\def\arraystretch{1.1}
	 {
 \caption{Computational evaluation for Model~\ref{model:kparties} and Model~\ref{model:reduced} with $k$ = $3$}\label{tab:spinglass}
 \begin{tabular}{lrrrr|rrrr}
\toprule
&\multicolumn{4}{c}{Model~\ref{model:kparties}}&\multicolumn{4}{c}{Model~\ref{model:reduced}}\\
\midrule
&\multicolumn{2}{c}{Root
node}&\multicolumn{2}{c}{Branch-and-bound}&\multicolumn{2}{c}{Root node}&\multicolumn{2}{c}{Branch-and-bound} \\
\noalign{\smallskip}
\midrule
\multicolumn{9}{c}{{\tt spinglass2g}} \\
\midrule
($k, \abs{V}$) & Time (s) & Gap (\%) & Time (s) & Nodes & Time (s) & Gap (\%) & Time (s) & Nodes \\
\midrule
($3, 3\times 3$) & 0.03 & 0.00 & 0.03 & 0 & 0.03 & 0.00 & 0.03 & 0 \\
($3, 4\times 4$) & 0.20 & 0.00 & 0.34 & 0 & 0.02 & 0.00 & 0.38 & 0 \\
($3, 5\times 5$) & 1.42 & 0.00 & 5.54 & 0 & 0.09 & 0.17 & 0.16 & 0 \\
($3, 6\times 6$) & 8.12 & 0.00 & 9.69 & 0 & 0.17 & 0.00 & 0.30 & 0 \\
($3, 7\times 7$) & 25.32 & 0.48 & 14075.4 & 4 & 0.76 & 0.98 & 10.80 & 0 \\
($3, 8\times 8$) & 127.13 & 0.00 & - & - & 2.05 & 0.00 & 2.52 & 0 \\
($3, 9\times 9$) & - & - & - & - & 4.66 & 0.91 & 40.98 & 0 \\
($3, 10\times 10$) & - & - & - & - & 6.19 & 0.00 & 6.74 & 0 \\
($3, 11\times 11$) & - & - & - & - & 5.00 & 0.00 & 46.11 & 0 \\
($3, 12\times 12$) & - & - & - & - & 7.86 & 0.00 & 9.82 & 0 \\
($3, 13\times 13$) & - & - & - & - & 14.02 & 0.00 & 229.40 & 0 \\
($3, 14\times 14$) & - & - & - & - & 16.12 & 0.00 & 23.13 & 0 \\
($3, 15\times 15$) & - & - & - & - & 19.57 & 0.03 & 1420.8 & 0 \\
($3, 16\times 16$) & - & - & - & - & 16.01 & 0.18 & 58900 & 28 \\
\midrule
\multicolumn{9}{c}{{\tt spinglass2pm}} \\
\midrule
($3, 3\times 3$) & 0.02 & 0.00 & 0.05 & 0 & 0.03 & 0.00 & 0.04 & 0 \\
($3, 4\times 4$) & 0.20 & 0.00 & 0.26 & 0 & 0.05 & 0.00 & 0.06 & 0 \\
($3, 5\times 5$) & 1.72 & 0.00 & 1.93 & 0 & 0.09 & 0.00 & 0.12 & 0 \\
($3, 6\times 6$) & 6.89 & 2.20 & 8.89 & 0 & 0.21 & 2.31 & 0.42 & 0 \\
($3, 7\times 7$) & 26.53 & 2.00 & 43.88 & 4 & 0.77 & 2.10 & 6.96 & 0 \\
($3, 8\times 8$) & 81.55 & 0.00 & - & - & 2.05 & 0.00 & 4.97 & 0 \\
($3, 9\times 9$) & & - & - & - & 2.03 & 1.21 & 31.52 & 0 \\
($3, 10\times 10$) & & - & - & - & 2.78 & 0.16 & 21.11 & 0 \\
($3, 11\times 11$) & & - & - & - & 4.28 & 0.00 & 6152.6 & 28 \\
($3, 12\times 12$) & & - & - & - & 6.73 & 0.00 & 18.27 & 0\\
($3, 13\times 13$) & & - & - & - & 9.61 & 1.08 & 859.37 & 0 \\
($3, 14\times 14$) & & - & - & - & 12.63 & 1.18 & - & -\\
\bottomrule
\end{tabular}
}
\end{center}
\end{table}
\subsubsection*{Results on Model~\ref{model:reduced} and Model~\ref{model:node-edge}}
Let us now compare the results of Models~\ref{model:reduced},~\ref{model:node-edge} presented in
Table~\ref{tab:compare_node-edge}. First, for {\tt band} instances, the
performance of Model~\ref{model:reduced} is significantly better than
Model~\ref{model:node-edge}. This is largely because Model~\ref{model:reduced}
has much less binary variables and constraints due the small sizes of its maximal clique sets.
Second, for {\tt spinglass2g} problem instances, the performance of
Models~\ref{model:reduced} and~\ref{model:node-edge} are similar. When $k$ or
the sizes of instances get larger, Model~\ref{model:reduced} is less attractive
than Model~\ref{model:node-edge}. This is probably because the sizes of the maximal clique sets are large, causing a great number of clique inequalities.
Third, the strength of Model~\ref{model:reduced} is generally much stronger than that
of~\ref{model:node-edge}. This is illustrated by instances of~{\tt spinglass2g}, where the continuous relaxation of Model~\ref{model:reduced} leads to $0$ optimality gap. 

\begin{table}[h!]
 \begin{center}
 	\def\arraystretch{1.1}
	 {
 \caption{Continuous relaxations of Model~\ref{model:node-edge} and Model~\ref{model:reduced}}
 \label{tab:compare_node-edge}
 \begin{tabular}{lrrrr|rrrr}
\toprule
&\multicolumn{4}{c}{Model~\ref{model:node-edge}}&\multicolumn{4}{c}{Model~\ref{model:reduced}}\\
\midrule
&\multicolumn{2}{c}{Root
node}&\multicolumn{2}{c}{Branch-and-bound}&\multicolumn{2}{c}{Root node}&\multicolumn{2}{c}{Branch-and-bound} \\
\midrule
\noalign{\smallskip}
($k, \abs{V}$) & Time (s) & Gap (\%) & Time (s) & Nodes & Time (s) & Gap (\%) & Time (s) & Nodes \\
\midrule
\multicolumn{9}{c}{{\tt band}} \\
\midrule
($3, 50$) & 0.09 & 100.00 & 1588.65 & 345246 & 0.03 & 5.38 & 0.28 & 0 \\
($3, 100$) & 0.25 & 100.00 & -  & -  & 0.04 & 5.47 & 0.76 & 0 \\
($3, 150$) & 0.51 & 98.64 & -  & -  & 0.06 & 4.78 & 1.58 & 0 \\
($3, 200$) & 0.35 & 100.00 & -  & -  & 0.07 & 5.51 & 5.14 & 0 \\
($3, 250$) & 0.47 & 100.00 & -  & -   & 0.08 & 5.52 & 5.61  & 0 \\
($4, 50$) & 0.09 & 107.14& -    & -   & 0.04 & 12.21 & 4.86  & 0 \\
($4, 100$) & 0.25 & 111.40 & -    & -   & 0.08 & 13.39 & 13.05 & 0 \\
($4, 150$) & 0.38 & 112.79 & -    & -   & 0.09 & 13.78 & 36.36 & 0 \\
($4, 200$) & 0.35 & 112.55 & -    & -   & 0.16 & 13.59 & 70.74 & 122 \\
($4, 250$) & 0.44 & 113.14 & -    & -   & 0.21 & 13.78 & 112.85 & 254 \\
\midrule
\multicolumn{9}{c}{{\tt spinglass2g}} \\
\midrule
($3, 10\times 10$) & 0.06 & 3.79 & 10.28 & 58  & 6.19  & 0.00 & 6.74  & 0 \\
($3, 11\times 11$) & 0.07 & 8.5  & 17.63 & 145  & 5.00  & 0.00 & 46.11 & 0 \\
($3, 12\times 12$) & 0.09 & 10.61 & 16.00 & 89  & 7.86  & 0.00 & 9.82  & 0 \\
($3, 13\times 13$) & 0.11 & 9.41 & 34.01 & 584  & 14.02 & 0.00 & 229.40 & 0 \\
($3, 14\times 14$) & 0.13 & 9.01 & 67.20 & 2126 & 16.12 & 0.00 & 23.13 & 0 \\
($3, 15\times 15$) & 0.15 & 10.01 & 77.30 & 860  & 19.57 & 0.03 & 1420.8 & 0 \\
($4, 10\times 10$) & 0.07 & 10.38 & 27.00 & 3075 & 6.99 & 0.00 & 22.70 & 0 \\
($4, 11\times 11$) & 0.09 & 7.50 & 30.50 & 3654 & 15.212 & 0.00 & 50.47 & 0 \\
($4, 12\times 12$) & 0.11 & 9.50 & 75.40 & 3068 & 24.696 & 0.00 & 144.87 & 0 \\
($4, 13\times 13$) & 0.14 & 9.19 & 79.03 & 4211 & 76.670 & 0.00 & -   & - \\
($4, 14\times 14$) & 0.17 & 8.10 & 42.68 & 882  & 76.940 & 0.00 & -   & - \\
($4, 15\times 15$) & 0.34 & 7.20 & 430.39 & 11314 & 279.98 & 0.00 & -   & - \\
\bottomrule
\end{tabular}
}
\end{center}
\end{table}

\subsubsection*{Results on Model~\ref{model:SDP1} and Model~\ref{model:SDP1reduce}}
We now compare the performance of Model~\ref{model:SDP1} and Model~\ref{model:SDP1reduce} with respect to their continuous relaxation values
and solution time. Since the problem is a minimisation problem, the higher the value is, the stronger is the lower bound. 
The numerical results are summarized in Table~\ref{tab:sdpcompare}.

Overall, we remark that that the continuous relaxation of Model~\ref{model:SDP1reduce} reduces significantly the computational time compared with that of Model~\ref{model:SDP1}, though a bit inferior in the solution quality. In addition, as expected, the solution time for Model~\ref{model:SDP1reduce} grows linearly with respect to the size of the graph while the computational time for Model~\ref{model:SDP1} grows more significantly as the instance size increases.

\begin{table}[h!]
 \begin{center}
 	\def\arraystretch{1.1}
	 {
 \caption{Continuous relaxations for Model~\ref{model:SDP1} and
 Model~\ref{model:SDP1reduce} }\label{tab:sdpcompare}
 \begin{tabular}{lrrr|rr}
\toprule
\multicolumn{6}{c}{\tt spinglass2g} \\
\midrule
&\multicolumn{2}{c}{Model~\ref{model:SDP1}}&&\multicolumn{2}{c}{Model~\ref{model:SDP1reduce}}\\
\midrule
($k, \abs{V}$) & Time (s) & Value&& Time (s) & Value \\
\midrule
($3, 11\times 11$) & 125.84 & -8.41227e+06 & & 1.74 & -8.42335e+06 \\
($3, 12\times 12$) & 359.43 & -1.06568e+07 & & 2.29 & -1.06777e+07 \\
($3, 13\times 13$) & 866.25 & -1.20763e+07 & & 4.00 & -1.2088e+07 \\
($3, 14\times 14$) & 2215.69 & -1.4147e+07 & & 5.64 & -1.416e+07  \\
($3, 15\times 15$) & 5072.07 & -1.80742e+07 & & 10.33 & -1.80904e+07 \\
($4, 11\times 11$) & 130.806 & -8.44258e+06 & & 1.88 & -8.45242e+06 \\
($4, 12\times 12$) & 343.846 & -1.07031e+07 & & 2.43 & -1.07238e+07 \\
($4, 13\times 13$) & 1164.69 & -1.2127e+07 & & 4.34 & -1.214e+07  \\
($4, 14\times 14$) & 2820.33 & -1.42114e+07 & & 5.59 & -1.42234e+07 \\
($4, 15\times 15$) & 6951.17 & -1.81605e+07 & & 10.88 & -1.81734e+07 \\
\midrule
\multicolumn{6}{c}{\tt spinglass2pm} \\
\midrule
($3, 11\times 11$) & 876.87  & -107.67 & & 1.21 & -108.02 \\
($3, 12\times 12$) & 2227.27 & -129.40 & & 1.66 & -129.83\\
($3, 13\times 13$) & 5500.92 & -150.58 & & 2.54 & -151.24\\
($3, 14\times 14$) & 11349.41 & -175.72 & & 3.37 & -176.41\\
($3, 15\times 15$) & 33188.43 & -197.88 & & 5.67 & -198.66 \\
($4, 11\times 11$) & 102.62  & -108.70 & & 1.26 & -109.16 \\
($4, 12\times 12$) & 291.59  & -130.23 & & 1.73 & -130.64\\
($4, 13\times 13$) & 729.29  & -152.08 & & 2.80 & -152.68\\
($4, 14\times 14$) & 1905.11 & -177.43 & & 3.56 & -178.11\\
($4, 15\times 15$) & 4791.15 & -199.90 & & 6.67 & -200.67 \\
\bottomrule
\end{tabular}
}
\end{center}
\end{table}

\subsubsection*{Results on Model~\ref{model:reduced} and Model~\ref{model:SDP1reduce}}
We compare the strength of Model~\ref{model:reduced} and
Model~\ref{model:SDP1reduce} with respect to their continuous relaxation values
and solution time. Again the problem is a minimization problem, the higher the
value is, the stronger is the lower bound. 

Our previous numerical results show that the computational time and bounds of
both Model~\ref{model:reduced} and Model~\ref{model:SDP1reduce} are quite
similar for test instances of {\tt band, spinglass2g, spinglass2pm}. To contrast
these two compact models, we present the numerical results for tests instances
of {\tt rndgraph}.  The numerical results are summarized in
Table~\ref{tab:compare_reduced}. We outline the following key observations. 
\begin{enumerate}
\item The strength of the continuous relaxation of Model~\ref{model:SDP1reduce}
    neither dominates nor is dominated by that of~\ref{model:reduced}. See, for
    instance, problem instance $(3, 40)$ and $(3, 60)$. 
\item For larger problem instances ($\abs{V} \ge 100$),
    Model~\ref{model:SDP1reduce} appear much more attractive than
    Model~\ref{model:reduced} in terms of computational scalability and bound
    quality. 
\end{enumerate}

\begin{table}[h!]
 \begin{center}
 	\def\arraystretch{1.1}
	 {
 \caption{Continuous relaxations of Model~\ref{model:reduced} and
 Model~\ref{model:SDP1reduce} with $k=3$ ({\tt rndgraph}).}\label{tab:compare_reduced}
 \begin{tabular}{lrr|rrr}
\toprule
&\multicolumn{2}{c}{Model~\ref{model:reduced}}&\multicolumn{2}{c}{Model~\ref{model:SDP1reduce}}\\
\midrule
\noalign{\smallskip}
($k, \abs{V}$) & Time (s) & Value&& Time (s) & Value \\
\midrule
($3, 10$) & 0.01  & 0.00 & & 0.06 & 0.00 \\
($3, 20$) & 0.00  & 0.00 & & 0.03 & 0.00 \\
($3, 30$) & 0.01  & 0.00 & & 0.03 & 0.00 \\
($3, 40$) & 0.07  & 0.67 & & 0.05 & 0.14 \\
($3, 50$) & 0.61  & 0.00 & & 0.12 & 0.00 \\
($3, 60$) & 1.97  & 1.11 & & 0.25 & 0.41 \\
($3, 70$) & 3.75  & 4.01 & & 0.49 & 2.53 \\
($3, 80$) & 12.44 & 4.94 & & 1.09 & 3.47 \\
($3, 90$) & 24.46 & 10.96 & & 1.69 & 9.86 \\
($3, 100$) & 49.76 & 14.67 & & 3.73 & 16.45 \\
($3, 110$) & 85.13 & 23.53 & & 5.62 & 26.11 \\
($3, 120$) & 209.22 & 28.86 & & 12.21 & 35.53 \\
($3, 130$) & 326.67 & 37.00 & & 21.07 & 48.57 \\
($3, 140$) & -   & -   & & 34.40 & 61.60 \\
($3, 150$) & -   & -   & & 48.33 & 78.47 \\
\bottomrule
\end{tabular}
}
\end{center}
\end{table}
\section{Conclusion}\label{sec:conclude}
This work introduces two compact reformulations of the minimum $k$-partition
problem exploiting the structured sparsity of the underlying graph. The first
model is a binary linear program while the second is an integer semidefinite
program. Both are based on the maximal clique set corresponding to the chordal
extension of the original graph. Numerical results show that the proposed models
numerically dominate state-of-the-art formulations.

Based on the results presented in this paper, several research directions can be
considered. First, alternative algorithms may be implemented to obtain optimal
clique sets minimizing the number of integer variables. Second, separation
algorithms for valid inequalities can be investigated. Third, specialized
branch-and-bound algorithms for the compact SDP Model~\ref{model:SDP1reduce} can
be considered. Lastly, combining Models~\ref{model:reduced}
and~\ref{model:node-edge} to obtain a novel integer LP formulation is also left for future research. 

\begin{acknowledgement}
This research was partly funded by the Australia Indonesia Centre.
\end{acknowledgement}
\bibliographystyle{spmpsci} 
\bibliography{mink}
\end{document}